\documentclass[11pt,oneside,english]{amsart}
\usepackage[T1]{fontenc}
\usepackage[latin9]{inputenc}
\usepackage{geometry}
\usepackage{babel}
\usepackage{amsthm}
\usepackage{amsbsy}
\usepackage{amssymb}
\usepackage{graphicx}
\usepackage{graphics}
\usepackage[unicode=true,pdfusetitle,
 bookmarks=true,bookmarksnumbered=false,bookmarksopen=false,
 breaklinks=false,pdfborder={0 0 1},backref=false,colorlinks=false]
 {hyperref}

\def\AT{\mathbf{o}}
\def\CF{\mathit{CF}}

\def\EH{\mathit{EH}}
\def\HF{\mathit{HF}}
\def\HFh{\widehat{\HF}}
\def\SFH{\mathit{SFH}}
\def\NN{\mathbb{N}}
\def\bolda{\boldsymbol{\alpha}}
\def\boldb{\boldsymbol{\beta}}
\def\arca{\underline{\mathbf{a}}}
\def\arcb{\underline{\mathbf{b}}}
\def\a{\alpha}
\def\b{\beta}
\def\T{\mathbb{T}}
\def\x{\mathbf{x}}
\def\y{\mathbf{y}}
\def\z{\mathbf{z}}
\def\S{\Sigma}

\def\Z{\mathbb{Z}}

\def\ol{\overline}
\def\H{\mathcal{H}}
\def\F{\mathcal{F}}
\def\N{\mathbb{N}}
\def\Int{\text{Int}}
\def\cP{\mathcal{P}}

\makeatletter

\providecommand{\tabularnewline}{\\}

\numberwithin{equation}{section}
\numberwithin{figure}{section}
 \theoremstyle{definition}
 \newtheorem{defn}{\protect\definitionname}
\theoremstyle{plain}
\newtheorem{thm}{\protect\theoremname}
\newtheorem{prop}[thm]{Proposition}
  \theoremstyle{plain}
  \newtheorem{lem}[thm]{\protect\lemmaname}
  \theoremstyle{remark}
  \newtheorem{rem}[thm]{Remark}
  \newtheorem*{rem*}{\protect\remarkname}
  \newtheorem{qn}{Question}
  \theoremstyle{plain}
  \newtheorem{cor}[thm]{\protect\corollaryname}

\makeatother

  \providecommand{\corollaryname}{Corollary}
  \providecommand{\definitionname}{Definition}
  \providecommand{\lemmaname}{Lemma}
  \providecommand{\remarkname}{Remark}
\providecommand{\theoremname}{Theorem}

\begin{document}

\title{Spectral order for contact manifolds with convex boundary}%

\author{Andr\'as Juh\'asz}%
\address{Mathematical Institute, University of Oxford, Andrew Wiles Building,
Radcliffe Observatory Quarter, Woodstock Road, Oxford, OX2 6GG, UK}%
\email{juhasza@maths.ox.ac.uk}%

\author{Sungkyung Kang}%
\address{Mathematical Institute, University of Oxford, Andrew Wiles Building,
Radcliffe Observatory Quarter, Woodstock Road, Oxford, OX2 6GG, UK}%
\email{sungkyung.kang@maths.ox.ac.uk}%


\date{}
\begin{abstract}
We extend the Heegaard Floer homological definition of spectral order for closed contact 3-manifolds
due to Kutluhan, Mati\'c, Van Horn-Morris, and Wand to contact 3-manifolds with
convex boundary. We show that the order of a codimension zero contact submanifold
bounds the order of the ambient manifold from above.
As the neighborhood of an overtwisted disk has order zero, we obtain
that overtwisted contact structures have order zero.
We also prove that the order of a small perturbation of a
$2\pi$ Giroux torsion domain has order at most two,
hence any contact structure with positive Giroux torsion has order at most two
(and, in particular, a vanishing contact invariant).
\end{abstract}

\maketitle

\section{Introduction}

Algebraic torsion of closed contact $(2n-1)$-manifolds was defined by
Latschev and Wendl~\cite{key-11} via symplectic field theory.
It is an invariant with values in $\NN \cup \{\infty\}$ whose finiteness
gives obstructions to the existence of symplectic fillings and exact symplectic cobordisms.
They also showed that the order of algebraic torsion is zero if and only if the
contact homology is trivial -- in particular, if the contact structure is overtwisted --
and it has order at most one in the presence of positive Giroux torsion.
Note that the analytical foundations of symplectic field theory are still under
development. Hence, in the appendix, Hutchings provided a similar numerical invariant
for contact 3-manifolds via embedded contact homology; however, it is currently unknown
whether this is independent of the contact form.

Motivated by the isomorphism between embedded contact homology
and Heegaard Floer homology, Kutluhan, Mati\'c, Van Horn-Morris, and Wand~\cite{key-10, key-14}
defined a Heegaard Floer homological analogue of algebraic torsion for closed contact 3-manifolds
called \emph{spectral order} (or \emph{order} in short), and denoted it by~$\AT$.
Their definition uses open book decompositions, and gives a refinement of the
Ozsv\'ath-Szab\'o contact invariant~$c(\xi)$. Using the fact that an overtwisted contact structure
is supported by an open book with non right-veering monodromy, they proved that $\AT(M,\xi) = 0$ if
$\xi$ is overtwisted.

In this paper, we extend $\AT$ to contact manifolds with convex boundary,
following the definition of Kutluhan et al.~in the closed case.
The definition is in terms of a partial open book decomposition of
the underlying sutured manifold supporting the contact structure,
and a collection of arcs on the page, containing a basis. This data gives rise to a filtration of
the sutured Floer boundary map, and the spectral order is the index of the first
page of the associated spectral sequence where the distinguished generator
representing the contact invariant vanishes, or $\infty$ otherwise. Then
we take the minimum over all partial open books together with a collections of arcs containing a basis.
(This extension of the definition of~$\AT$ was also independently observed
by Kutluhan et al.~\cite{key-14}.)

Our first main result is that the spectral order of a codimension zero contact submanifold
gives an upper bound on the spectral order of the ambient manifold.

\begin{thm} \label{thm:ineq}
Let $(M,\xi)$ be a contact $3$-manifold with convex boundary.
If $(N,\xi|_N)$ is a codimension zero submanifold of~$\text{Int}(M)$ with convex boundary,
then
\[
\AT(N,\xi|_N) \ge \AT(M,\xi).
\]
\end{thm}

We will prove this result in Section~\ref{sec:ineq}.
As a corollary, we show that if a contact manifold with convex boundary
is overtwisted, then it has spectral order zero. This follows
immediately from a simple computation that a neighborhood of an overtwisted
disk has spectral order zero.

In Section~\ref{sec:gt}, we carry out a computation that shows that the spectral order
of a slight enlargement of a Giroux $2\pi$-torsion $T^2 \times I$
has spectral order at most two. In particular, every contact manifold
with positive Giroux torsion has vanishing Ozsv\'ath-Szab\'o invariant,
which was proved in the closed case by Ghiggini, Honda, and Van Horn-Morris~\cite{key-5}
(the sutured case also follows from their work when combined with \cite[Theorem~1.1]{key-2}).
Together with Theorem~\ref{thm:ineq}, we obtain the following corollary.

\begin{thm} \label{thm:gt}
If a contact $3$-manifold $(M,\xi)$ with convex boundary has Giroux $2\pi$-torsion,
then
\[
\AT(M,\xi) \le 2.
\]
\end{thm}

The inequality $\mathit{AT} \le 1$ was shown in the closed case by Latschev and Wendl~\cite[Theorem~2]{key-11}
via symplectic field theory, and conjectured in the Heegaard Floer setting
in the closed case by Kutluhan et al.~\cite[Question~6.3]{key-14}.
More generally, they asked whether the presence of planar $k$-torsion (see~\cite[Section~3.1]{key-11}
for a definition) implies that the spectral order is at most~$k$.

\subsection*{Acknowledgement} We would like to thank Cagatay Kutluhan, Gordana Mati\'c,
Jeremy Van Horn-Morris, and Andy Wand for pointing out a mistake in the first version of
this paper, and for helpful discussions, and Paolo Ghiggini and Chris Wendl for their comments.

This project has received funding from the European Research Council (ERC) under the European
Union's Horizon 2020 research and innovation programme (grant agreement No 674978).
The first author was supported by a Royal Society Research Fellowship.

\section{Spectral order for manifolds with boundary}

We first recall the definition of spectral order
for closed contact $3$-manifolds due to Kutluhan, Mati\'c, Van Horn-Morris, and Wand~\cite{key-14}.
Let $(M,\xi)$ be a closed contact 3-manifold. By the Giroux correspondence
theorem \cite{key-9}, the contact structure $\xi$ is supported by
some open book decomposition~$(S,\phi)$ of $M$, which is well-defined up to
positive stabilizations. In particular, $M$ is identified with 
$S \times I/{\sim}$, where $(x,1) \sim (\phi(x),0)$ for every $x \in S$,
and $(x,t) \sim (x,t')$ for every $x \in \partial S$ and $t$, $t' \in I$.

An arc basis on~$S$ is a set of pairwise disjoint properly embedded arcs
that forms a basis of $H_{1}(S,\partial S)$.
A collection of pairwise disjoint arcs $\underline{\mathbf{a}} = \{a_1,\dots,a_n\}$
on $S$ that contains a basis induces an ``overcomplete'' Heegaard diagram
$(\Sigma,\boldsymbol{\alpha},\boldsymbol{\beta})$ of $M$, as follows.
We obtain~$\arcb = \{b_1,\dots,b_n\}$ by isotoping $\arca$
such that the endpoints of $\arca$ are moved in
the positive direction along $\partial S$, and $|a_i \cap b_j| = \delta_{ij}$ for $i$, $j \in \{1, \dots n\}$.
Then we set $\Sigma = (S\times\{1/2\}) \cup_{\partial S} (-S\times\{0\})$. Furthermore, we let
$\bolda = \{\a_1, \dots, \a_n\}$ and $\boldb = \{\b_1,\dots,\b_n\}$, where
\[
\begin{split}
\a_i &:= (a_i \times\{1/2\}) \cup (a_i \times \{0\}) \text{ and}\\
\b_i &:= (b_i \times \{1/2\}) \cup (\phi(b_i) \times \{0\})
\end{split}
\]
for $i \in \{1, \dots, n \}$.
We also choose a basepoint in each connected component of $S \setminus \arca$
away from the isotopy between $\arca$ and $\arcb$, and denote the set of these on~$S \times \{1/2\} \subset \S$ by~$\mathbf{z}$.
Then $(\Sigma,\boldsymbol{\alpha},\boldsymbol{\beta},\mathbf{z})$
is a multi-pointed Heegaard diagram of~$M$.

We say that a domain $D$ in the diagram $(\Sigma,\boldsymbol{\alpha},\boldsymbol{\beta},\mathbf{z})$
connects $\mathbf{x}$, $\mathbf{y}\in\mathbb{T}_{\alpha}\cap\mathbb{T}_{\beta}$
if $\partial(\partial D\cap\boldsymbol{\alpha})=\mathbf{x}-\mathbf{y}$
and $\partial(\partial D\cap\boldsymbol{\beta})=\mathbf{y}-\mathbf{x}$,
and $n_z(D) = 0$ for every $z \in \mathbf{z}$.
We denote by $D(\mathbf{x},\mathbf{y})$ the set of such domains.
If $\x = (x_1, \dots, x_n)$, then there is a unique permutation $\pi_\x \in S_n$
such that $x_i \in \a_i \cap \b_{\pi_\x(i)}$ for every $i \in \{1, \dots, n\}$.
Using the above Heegaard diagram, Kutluhan at al.~\cite{key-14} defined
a function~$J_{+}$ that assigns an integer
\[
J_{+}(D)=n_{\mathbf{x}}(D)+n_{\mathbf{y}}(D)-e(D)+|\mathbf{x}|-|\mathbf{y}|
\]
to every domain $D\in D(\mathbf{x},\mathbf{y})$.
Here, $n_{\mathbf{x}}(D)$ is the sum over all $p\in\mathbf{x}$ of
the averages of the coefficients of $D$ at the four regions around
$p$, the term $e(D)$ is the Euler measure of $D$, and $|\mathbf{x}|$,
$|\mathbf{y}|$ are the number of cycles in the permutations $\pi_\x$ and $\pi_\y$, respectively. When
$D$ is a domain of Maslov index~$1$, the equality $e(D)=1-n_{\mathbf{x}}(D)-n_{\mathbf{y}}(D)$
holds by the work of Lipshitz~\cite{key-6}, so the formula becomes
\[
J_{+}(D)=2(n_{\mathbf{x}}(D)+n_{\mathbf{y}}(D))-1+|\mathbf{x}|-|\mathbf{y}|.
\]
For any topological Whitney disk $\psi \in \pi_{2}(\mathbf{x},\mathbf{y})$,
we can define $J_{+}(\psi)$ as the value $J_{+}(D(\psi))$, where $D(\psi)$
is the domain of $C$. The function $J_{+}$ is additive in the sense
that
\[
J_{+}(D_{1}+D_{2})=J_{+}(D_{1})+J_{+}(D_{2})
\]
for every $D_{1} \in D(\x_1,\x_2)$ and $D_{2} \in D(\x_2, \x_3)$.
Furthermore, $J_{+}(u)$ is always a nonnegative even integer for any J-holomorphic
disk~$u$. Hence, we have a splitting
\[
\widehat{\partial}_{\HF}=\partial_{0}+\partial_{1}+\partial_{2}+\cdots
\]
of the Heegaard Floer differential $\widehat{\partial}_{\HF}$, where $\partial_{i}$
is defined by counting all J-holomorphic disks $u$ satisfying $\mu(u)=1$
and $J_{+}(u)=2i$. As shown in~\cite{key-14}, this gives a spectral
sequence
\[
E^k(S,\phi,\arca)= H_{\ast}\left(E^{k-1}(S,\phi,\arca), d^{k-1}\right),
\]
induced by the filtered complex
\[
\left(C = \bigoplus_{i\in \N}\widehat{CF}(\Sigma,\boldb,\bolda,\mathbf{z})_{i} \text{, }\widehat{\partial}\right).
\]
The $j$-th coordinate of the differential $\widehat{\partial} \underline{c}$ for $\underline{c} = (c_{i})_{i \in \N} \in C$
and $j \in \mathbb{N}$ is defined as
\[
\left(\widehat{\partial}\underline{c}\right)_j = \sum_{i=0}^\infty \partial_{i} c_{i+j},
\]
and the filtration is given by
\[
\F_p C =  \bigoplus_{i = 0}^p \widehat{CF}(\Sigma,\boldb,\bolda,\mathbf{z})_i.
\]
Note that here we deviate slightly from the definition of Kutluhan et al.~\cite{key-14}
in that we take the direct sum defining~$C$ over~$\N$ instead of~$\Z$, but as we shall see,
the arising notion of spectral order is exactly the same.

Recall that a filtered complex
\[
\dots \subseteq \F_{p-1}C \subseteq \F_p C \subseteq \F_{p+1} C \subseteq \dots
\]
induces a spectral sequence by setting
\[
\begin{split}
Z^k_p &= \{\, x \in \F_p C \,\colon\, \partial x \in \F_{p-k} C\,\} \text{ and} \\
B^k_p &= \F_p C \cap \partial\F_{p+k} C.
\end{split}
\]
For $k \in \N$, the \emph{$k$-page} is the complex $\left(E^k = \bigoplus_{p \in \Z} E^k_p, d^k \right)$,
where
\[
E^k_p = \frac{Z^k_p}{Z^{k-1}_{p-1} + B^{k-1}_p},
\]
and the differential $d^k \colon E^k_p \to E^k_{p-k}$ is induced by the differential $\partial$ on the complex~$C$.

For an open book decomposition $(S,\phi)$ supporting $\xi$, and a
collection of arcs $\underline{\mathbf{a}}$ on~$S$ containing a
basis, we denote the induced spectral sequence defined above by $E^{n}(S,\phi,\arca)$.
Then note that, for every $k \in \Z_{>0}$,
\[
Z^k_0(S,\phi,\arca) = \{\, (c_i)_{i \in \N} \,\colon\, c_i = 0 \text{ for $i > 0$ and } \partial_0 c_0 = 0 \,\}.
\]

Recall that the contact element $\EH(\xi) \in \HFh(-M)$ is defined as the homology class of the intersection point
\[
\x_\xi := (\arcb \cap \arca) \times \{1/2\} \in \mathbb{T}_{\beta} \cap \mathbb{T}_{\alpha},
\]
where $\arca \times \{1/2\}$ and $\arcb \times \{1/2\}$ are subsets of $\bolda$ and $\boldb$, respectively, by definition.
As there are no non-trivial pseudo-holomorphic disks emanating from $\x_\xi$
in $(\S,\boldb,\bolda)$ that contribute to $\widehat{\partial}_{\HF}$,
it follows that $\partial_i \x_\xi = 0$ for every $i \in \N$.
We often view $\x_\xi$ as an element of~$C$ supported in degree zero; i.e.,
as a sequence $(d_i)_{i \in \N}$ such that $d_0 = \x_\xi$ and $d_i = 0$ for $i > 0$.
As such, $\x_\xi \in Z^k_0(S,\phi,\arca)$ for every $k \in \N$.
The following is \cite[Definitions~2.1 and~2.2]{key-14}.

\begin{defn}
Let $(M,\xi)$ be a closed contact $3$-manifold.
We say that $\AT(S,\phi,\arca) = k$ if the distinguished generator
\[
\x_\xi \in \widehat{\CF}(\Sigma,\boldb,\bolda,\mathbf{z})_{0},
\]
viewed in degree~0, is nonzero in $E^{k}(S,\phi,\arca)$, and zero
in~$E^{k+1}(S,\phi,\arca)$. Then we define the \emph{spectral order of $(M,\xi)$} as
\[
\AT(M,\xi)=\min\left\{\, \AT(S,\phi,\arca)\,\colon\, (S,\phi) \mbox{ supports } \xi
\mbox{ and } \arca \subset S \mbox{ contains an arc basis}\,\right\} .
\]
\end{defn}

Implicit in the above definition is the choice of an almost complex structure~$J$ on $\text{Sym}^g(\S)$.
Kutluhan et al.~\cite[Proposition~3.1]{key-14} showed that $\AT(S,\phi,\arca, J)$ is independent
of~$J$, hence we suppress it from our notation throughout.

\begin{rem*}
The contact element $\x_{\xi}$, viewed in degree zero,
vanishes in $E^{k+1}(S,\phi,\arca)$ if and only if it is contained in
\[
B_{0}^k(S,\phi,\arca) = \F_0 C \cap \widehat{\partial} \F_k C = \widehat{\CF}(\S,\boldb,\bolda)_0 \cap
\widehat{\partial} \left( \bigoplus_{i = 0}^k \widehat{\CF}(\S,\boldb,\bolda)_i \right).
\]
This holds precisely if there exist elements $c_{i} \in \widehat{CF}(\Sigma,\boldb,\bolda,\mathbf{z})$
for $i \in \{\, 0,\dots,k \,\}$ such that
\begin{equation} \label{eqn:main}
\begin{split}
\sum_{i=0}^{k}\partial_{i} c_{i} &= \x_\xi, \text{ and} \\
\sum_{i=0}^{k-j}\partial_{i}c_{i+j} &= 0 \text{ for all } j > 0.
\end{split}
\end{equation}
Indeed, if we set $c_i = 0$ for $i > k$, then the entries of $\widehat{\partial}(c_i)_{i \in \N}$
correspond to the left-hand side of equation~\eqref{eqn:main}, and so
this equation translates to $\widehat{\partial}(c_i)_{i \in \N} = (d_j)_{j \in \N}$,
where $d_0 = \x_{\xi}$ and $d_j = 0$ for $j > 0$. As equation~\eqref{eqn:main}
coincides with the one defining $\mathcal{B}^k(S,\phi,\arca)$ in~\cite[p.~5]{key-14},
it follows that it does not matter whether we take the direct sum over $\N$ or $\Z$
when we define~$\AT$.
\end{rem*}

Before extending this definition to manifolds with boundary, we first
review the definition of partial open book decompositions, introduced
by Honda, Kazez, and Mati\'{c}~\cite{key-2}. We follow the treatment
of Etgu and Ozbagci~\cite{key-4}. An abstract partial open book
decomposition is a triple $\mbox{\ensuremath{\mathcal{P}}}=(S,P,h)$,
where
\begin{itemize}
\item $S$ is a compact, oriented, connected surface with nonempty boundary,
\item $P=P_{1}\cup\dots\cup P_{r}$ is a proper subsurface of $S$ such
that $S$ is obtained from $\overline{S\setminus P}$ by successively
attaching $1$-handles $P_{1},\dots,P_{r}$,
\item $h:P\rightarrow S$ is an embedding such that $h|_{A}=\mbox{Id}_{A}$,
where $A=\partial P\cap\partial S$.
\end{itemize}
Given a partial open book decomposition $(S,P,h)$, we associate to
it a sutured $3$-manifold $(M,\Gamma)$, as follows. Let $H=S\times[-1,0]/{\sim}$,
where $(x,t)\sim(x,t')$ for every $x \in \partial S$ and $t$, $t'\in[-1,0]$.
Furthermore, let $N = P \times I/{\sim}$, where $(x,t)\sim(x,t')$ for
every $x\in A$ and $t$, $t'\in I$. We obtain the manifold~$M$
by gluing $(x,0)\in\partial N$ to $(x,0)\in\partial H$ and $(x,1)\in\partial N$
to $(h(x),-1)\in\partial H$ for every $x\in P$. The sutures are
defined as
\[
\Gamma=(\overline{\partial S\setminus\partial P})\times\{0\}\cup-(\overline{\partial P\setminus\partial S})\times\{1/2\}.
\]
Then
\[
\Sigma=(P\times\{0\}\cup-S\times\{-1\})/\sim
\]
is a Heegaard surface for $(M,\Gamma)$.

Let $\xi$ be a contact structure on $M$ such that $\partial M$
is convex with dividing set $\Gamma$. Similarly to the original Giroux
correspondence, we say that~$\xi$ is compatible with the partial
open book decomposition $(S,P,h)$ if
\begin{itemize}
\item $\xi$ is tight on the handlebodies $H$ and $N$,
\item $\partial H$ is a convex surface with dividing set $\partial S\times\{0\}$,
\item $\partial N$ is a convex surface with dividing set $\partial P\times\{1/2\}$.
\end{itemize}
Then the relative Giroux correspondence theorem says that $\xi$ is
uniquely determined up to contact isotopy, and given such a contact
structure $\xi$, any two partial open book decompositions compatible
with $\xi$ are related by positive stabilizations.

We now extend the definition of spectral order to manifolds with
boundary. Suppose that a contact $3$-manifold $(M,\xi)$ with convex
boundary $\partial M$ and dividing set $\Gamma$ is given. Then $(M,\Gamma)$
is a balanced sutured manifold if $M$ has no closed components. Indeed,
every convex surface has a non-empty dividing set. 
Furthermore, $\chi(R_{+}(\Gamma))=\chi(R_{-}(\Gamma))$
by \cite[Proposition~3.5]{polytope}. Then we have a compatible partial
open book decomposition $\mathcal{P} = (S,P,h)$.
An arc basis for $(S,P,h)$ is a set $\arca$ of
properly embedded arcs in~$P$ with endpoints on~$A$ such that
$S\setminus\arca$ deformation retracts onto $\overline{S\setminus P}$.
Similarly to the closed case, a partial open book decomposition of
$M$, together with a collection of pairwise disjoint arcs $\arca$
containing a basis and an appropriate choice of basepoints, gives
a multipointed sutured Heegaard diagram $(\Sigma,\bolda,\boldb,\mathbf{z})$
of $(M,\Gamma)$. Here, $\mathbf{z}$ consists of a basepoint in each component
of $P \setminus \arca$ disjoint from $\partial P \setminus \partial S$.

The differential~$\widehat{\partial}_{\SFH}$ of the sutured Floer chain complex counts the number of J-holomorphic
curves $u$ with $\mu(u)=1$, modulo the $\mathbb{R}$-action, that
do not intersect the suture $\Gamma=\partial\Sigma$ and the basepoints~$\z$. For any topological
Whitney disk $\psi$ from $\mathbf{x}\in\mathbb{T}_{\alpha}\cap\mathbb{T}_{\beta}$
to $\mathbf{y}\in\mathbb{T}_{\alpha}\cap\mathbb{T}_{\beta}$ that
does not intersect $\partial\Sigma$ and~$\z$, we define the number $J_{+}(\psi)$
as in the closed case by
\[
J_{+}(\psi)=n_{\mathbf{x}}(D)+n_{\mathbf{y}}(D)-e(D)+|\mathbf{x}|-|\mathbf{y}|,
\]
where $D=D(\psi)$ is the domain of $\psi$. Since the equality $e(D)=1-n_{\mathbf{x}}(D)-n_{\mathbf{y}}(D)$
for $\mu(D)=1$ still holds in the sutured case, we get that
\begin{equation}
J_{+}(\psi)=2(n_{\mathbf{x}}(D)+n_{\mathbf{y}}(D))-1+|\mathbf{x}|-|\mathbf{y}|\label{eqn:J}
\end{equation}
when $\mu(\psi)=1$. As in the closed case, the function $J_+$ is
clearly additive, and the same argument as in~\cite[Section~2.2]{key-14} shows
that $J_+(u)$ is a non-negative even integer for any $J$-holomorphic disk~$u$.

Hence, we can split the sutured Floer differential $\widehat{\partial}_{\SFH}$
as
\[
\widehat{\partial}_{\SFH} = \partial_{0} + \partial_{1}+\cdots,
\]
where $\partial_{r}$ counts J-holomorphic disks~$u$ with $\mu(u)=1$
and $J_{+}(u)=2r$.

Just like in the closed case, the pair
$\left(\bigoplus_{i \in \N} CF(\Sigma,\boldb,\bolda,\mathbf{z})_i \text{, } \widehat{\partial}\right)$,
where the map $\widehat{\partial}$ is defined as
\[
\widehat{\partial}(c_{i})_{i \in \N} = \left(\sum_{i=0}^\infty \partial_{i} c_{i+j}\right)_{j \in \N},
\]
is a filtered chain complex.
Using its induced spectral sequence, we can define the spectral order of $(M,\xi)$ in the following way.

\begin{defn}
For a contact $3$-manifold $(M,\xi)$ with
convex boundary, a compatible partial open book decomposition $\mathcal{P}$,
and a collection of pairwise disjoint arcs $\arca$ containing an arc basis, denote the induced
spectral sequence by $E^k(\mathcal{P},\underline{\mathbf{a}})$.
We say that $\AT(\mathcal{P},\arca)=k$ if the distinguished generator
$\x_\xi \in \CF(\S,\boldb,\bolda,\mathbf{z})_{0}$ in degree~0 remains
nonzero in~$E^k(\mathcal{P},\arca)$, but vanishes
in $E^{k+1}(\mathcal{P},\arca)$. Then we define the \emph{spectral order}
\[
\AT(M,\xi)= \min\left\{ \AT(\mathcal{P},\arca)\,\colon\,\mathcal{P}\mbox{ supports $\xi$ 
and }\arca \mbox{ contains a basis}\right\} .
\]
This is always a nonnegative integer.
\end{defn}
 
We will need the following lemma for for the proof of Theorem~\ref{thm:ineq}.
  
\begin{lem}\label{lem:ball}
Let $(M, \xi)$ be a contact $3$-manifold with (possibly empty) convex boundary,
and suppose that $B \subset \Int(M)$ is a tight contact ball. If $M_0 = M \setminus \Int(B)$
and $\xi_0 = \xi|_{M_0}$, then 
\[
\AT(M_0,\xi_0) \ge \AT(M,\xi).
\]
Furthermore, we have equality if $M$ is closed.
\end{lem}

\begin{proof}
  We denote by $\Gamma_B$ the dividing set of $\xi$ on $\partial B$.
  Let $\cP_0 = (S_0,P_0,h_0)$ be a partial open book decomposition of $(M_0,\gamma_0)$ 
  supporting $\xi_0$, together with a collection of arcs $\arca_0$ containing a basis,
  and write $(\S_0, \boldb_0, \bolda_0, \z_0)$ for the corresponding based sutured diagram
  of $(-M_0, -\Gamma_0)$.
    
  There is a disk component $D_B$ of $\overline{S_0 \setminus P_0}$ corresponding
  to $R_+(\Gamma_B)$, and a disk component $D_B'$ of $\overline{S \setminus h(P)}$
  corresponding to $R_-(\Gamma_B)$. Then $h$ uniquely extends to a diffeomorphism 
  \[
  h \colon P \cup D_B \to h(P) \cup D_B', 
  \]
  up to isotopy. If we set $S = S_0$ and $P = P \cup D_B$,
  then $\cP := (S,P,h)$ is a partial open book of $(M,\xi)$. Furthermore, $\arca = \arca_0$
  contains an arc basis for $\cP$. As $D_B$ lies in a component
  of $P \setminus \arca$ disjoint from $\partial P \setminus \partial S$,
  we need to add a basepoint $z_B$ here. The based diagram $(\S,\boldb,\bolda,\z)$ corresponding to $(\cP, \arca)$
  is obtained by filling in a boundary component of $\S_0$ with the disk $D_B \times \{0\}$, 
  and taking $\z = \z_0 \cup \{z_B\}$. Hence, $\AT(\cP,\arca) = \AT(\cP_0,\arca_0)$,
  as their defining filtered chain complexes agree. 
  It follows that $\AT(M_0,\xi_0) \ge \AT(M,\xi)$.
  
  Now suppose that $M$ is closed. Let $(S,h)$ be an arbitrary open book
  decomposition of $(M,\xi)$, and $\arca$ an arbitrary collection of arcs containing a basis.
  Then each component of $S \setminus \arca$ is homeomorphic to a disk; let $D$ be one of them.
  Consider the partial open book $\cP_0 = (S_0, P_0, h_0)$, where $S_0 = S$,
  $P_0 = S \setminus D$, and $h_0 = h|_{P_0}$. Then $\cP_0$ supports $(M_0,\xi_0)$.
  If $(\S,\boldb,\bolda,\z)$ is the diagram arising from $(S,h,\arca)$, and 
  $(\S_0,\boldb_0,\bolda_0,\z_0)$ is the diagram arising from $(\cP_0,\arca)$,
  then $\S = \S_0 \cup (D \times \{0\})$, $\boldb = \boldb_0$, $\bolda = \bolda_0$,
  and we obtain $\z_0$ by removing the unique point $\z \cap D$. Hence, 
  $\AT(S,h,\arca) = \AT(\cP_0,\arca)$. Since $(S,h,\arca)$ was arbitrary,
  we obtain that $\AT(M,\xi) \ge \AT(M_0,\xi_0)$. 
\end{proof}

Using Theorem~\ref{thm:ineq}, we will show that actually equality holds
in the first part of Lemma~\ref{lem:ball}.

\section{Inequality of spectral orders} \label{sec:ineq}

The goal of this section is to prove Theorem~\ref{thm:ineq} from the introduction.
We first briefly recall the construction of the contact gluing map $\Phi$ on sutured Floer homology,
defined by Honda, Kazez, and Mati\'c~\cite{key-2}.
Let $(M, \Gamma_M)$ be a sutured manifold, and let $(N, \Gamma_N)$ be a sutured submanifold of~$\text{Int}(M)$.
Furthermore, let $\xi$ be a contact structure on $M \setminus \text{Int}(N)$
with convex boundary and dividing set~$\Gamma_M$ on $\partial M$ and dividing set~$\Gamma_N$ on $\partial N$.
We can suppose that $M \setminus N$ has no isolated components; i.e., every
component of $M \setminus N$ intersects~$\partial M$. Indeed, by Lemma~\ref{lem:ball},
if we remove a tight contact ball from each isolated component, $\AT(M,\xi)$ does not decrease.

Choose a collar neighborhood $Z \simeq \partial N \times I$ of~$\partial N$
in~$M \setminus \text{Int}(N)$ such that $Z \cap N = \partial N \times \{0\}$,
on which the contact structure $\xi$ is $I$-invariant, and write $N' = M \setminus \text{Int}(N \cup Z)$.
Let $\Sigma_{N'}$ be a Heegaard surface
compatible with $\xi|_{N'}$, and let~$\Sigma_Z$ be a Heegaard
surface compatible with $\xi|_Z$. Then, for any sutured Heegaard
diagram $\H = (\Sigma,\boldb,\bolda)$
of $(N,\Gamma_N)$ that is contact-compatible near $\partial N$ in
the sense of Honda, Kazez, and Mati\'c~\cite{key-2}, the union $\Sigma\cup\Sigma_Z \cup \Sigma_{N'}$
is a Heegaard surface for~$(M,\Gamma_M)$, and we can complete $\bolda$ and $\boldb$ to attaching
sets of $(M,\Gamma_M)$ by adding $\bolda'$ and $\boldb'$ compatible with~$\xi|_{N' \cup Z}$. We write
\[
\H' = (\Sigma\cup\Sigma_Z \cup\Sigma_{N'},\boldb\cup\boldb',\bolda\cup\bolda').
\]
Then the map
\begin{eqnarray*}
\Phi_\xi \colon \CF(\H) & \rightarrow & \CF(\H'),\\
\y & \mapsto & (\y,\x')
\end{eqnarray*}
is a chain map, where $\x' \in \T_{\b'} \cap \T_{\a'}$
is the canonical representative of the contact class $\EH(\xi|_{N' \cup Z})$.
Note that this construction makes sense even if we replace Heegaard diagrams
with multipointed Heegaard diagrams.


\begin{proof}[Proof of Theorem~\ref{thm:ineq}]
As in the statement of Theorem~\ref{thm:ineq},
let $(M,\xi)$ be a contact $3$-manifold with convex boundary and dividing set~$\Gamma_M$,
and let $N$ be a codimension zero submanifold of~$\text{Int}(M)$, also with convex boundary and dividing set~$\Gamma_N$.
Then let $\mathcal{P}_{N} = (S_{N},P_{N},h_{N})$ be a partial open book decomposition of $(N,\xi|_{N},\Gamma_{N})$,
together with a choice of an arc basis~$\arca_N$, and let $\H = (\Sigma,\boldb,\bolda)$
be the corresponding diagram of $(-N,-\Gamma_N)$. 

Let $\zeta$ be an $I$-invariant contact structure on $\partial N \times I$ such that 
$\partial N \times \{t\}$ is convex with dividing set $\Gamma_N$ for every $t \in I$. 
According to the Remark after \cite[Lemma~4.1]{key-2}, we can first extend $\H$
to a diagram of $(-N \cup (\partial N \times [0,2]), -\Gamma_N \times \{2\})$ that is contact 
compatible near $\partial N \times \{2\}$, by gluing two Heegaard surfaces 
arising from certain special partial open book decompositions
of $(\partial N \times I, \zeta)$. We denote the resulting compact compatible diagram
$(\S \cup \S_\zeta, \boldb \cup \boldb_\zeta, \bolda \cup \bolda_\zeta)$. 
Then, using Step~2 of \cite[Section~4]{key-2}, and as explained above,
we can further extend this to a diagram
\[
\H' = (\S \cup \S_\zeta \cup \S_Z \cup \S_{N'}, \boldb \cup \boldb_\zeta \cup \boldb', \bolda \cup \bolda_\zeta \cup \bolda')
\] 
of $(-M,-\Gamma_M)$.
Analogously to the gluing map, we obtain a chain map
\begin{eqnarray*}
\Phi \colon \CF(\H) & \rightarrow & \CF(\H'),\\
\y & \mapsto & (\y,\x')
\end{eqnarray*}
where $\x' \in \T_{\b_\zeta \cup \b'} \cap \T_{\a_\zeta \cup \a'}$
is the canonical representative of the contact class $\EH(\xi|_{(\partial N \times [0,2]) \cup Z \cup N'})$.
As $\H$ is not necessarily contact compatible, we do not claim that $\Phi$ is the contact gluing map
under naturality, but this is not necessary for our purposes. 
By construction, $\Phi$ maps the contact class~$\x_{\xi|_N}$ to the contact class~$\x_\xi$.
Note that this construction of Honda, Kazez, and Mati\'c~\cite{key-2}
actually gives a partial open book $\mathcal{P} = (S,P,h)$ supporting $(M,\xi)$
and an arc basis~$\arca$ that extend $\mathcal{P}_N$ and~$\arca_N$, respectively.

Now consider the case when $\arca_N$ is not an
arc basis, but a collection of pairwise disjoint arcs that contains
an arc basis. Then we need to choose basepoints $\mathbf{z}$ such that
every connected component of $P_N \setminus \cup \arca_N$
that does not intersect $\partial P_N \setminus \partial S_N$ has exactly
one basepoint. The gluing process can be applied to this case without
modification, to get a collection of pairwise disjoint arcs~$\arca$
in~$P$. After gluing, every connected component of $P \setminus \cup \arca$
disjoint from $\partial P \setminus \partial S$
contains exactly one basepoint, since such a component must come from
$P_N \setminus \cup \arca_N$, and other components do not
contain a basepoint. Hence, the data $(\mathcal{P},\arca,\mathbf{z})$
satisfies the conditions needed to define its order. The proof of the
fact that the gluing map is a chain map between Floer chain complexes~\cite{key-2}
also applies to this case without further modification, for the same reason.

\begin{lem} \label{lem:morphism}
Let $\Phi$ be as above. Then the map
\[
\ol{\Phi} \colon \bigoplus_{i \in \N} \CF(\H)_i \to \bigoplus_{i \in \N} \CF(\H')_i
\]
defined by $\ol{\Phi}\left((c_i)_{i \in \N}\right) = (\Phi(c_i))_{i \in \N}$
is a filtered chain map, hence induces a morphism $(\Phi^k)_{k \in \N}$
of spectral sequences; i.e., $\Phi^0 = \ol{\Phi}$, and
\[
\Phi^k \colon E^k(\mathcal{P}_{N},\arca_{N}) \rightarrow E^k(\mathcal{P},\arca)
\]
is a chain map for every $k \in \N$ such that the map induced on homology is $\Phi^{k+1}$.
\end{lem}

\begin{proof}
Let $\x$, $\y \in \T_\b \cap \T_\a$. Any holomorphic disk~$u$ from
$(\x,\x^{\prime})$ to $(\y,\x^{\prime})$ in $\CF(\H')$
is actually a holomorphic disk from $\x$ to $\y$ in $\CF(\H)$;
i.e., its domain $D := D(u)$ is zero outside $\Sigma$; see~\cite[p.~12]{key-2}.
Since the Euler measure and the point measure of $D$ depend only on the non-zero coefficients,
the Maslov index of $u$ in $\H$ and in $\H'$ are the same. Suppose that $\mu(u) = 1$.
Then, in $\H'$, we have
\begin{eqnarray*}
J_{+}(u) & = & 2(n_{(\x,\x^{\prime})}(D)+n_{(\y,\x^{\prime})}(D)) -
1 + |(\x,\x^{\prime})|-|(\y,\x^{\prime})|\\
 & = & 2(n_{\x}(D)+n_{\y}(D))-1+|\x|+|\x^{\prime}|-|\y|-|\x^{\prime}|\\
 & = & 2(n_{\x}(D)+n_{\y}(D))-1+|\x|-|\y|.
\end{eqnarray*}
This is the same as the value of $J_{+}(u)$ in $\H$.
Hence $\Phi$ preserves the $J_{+}$ filtration.

Now, by the definition of the differential $\partial_i$,
the map $\Phi$ commutes with $\partial_i$ for all $i \in \N$.
Hence, it commutes with the total differential $\widehat{\partial}$,
and so $\ol{\Phi}$ is a filtered chain map.
Therefore $\ol{\Phi}$ induces a morphism $(\Phi^k)_{k \in \N}$
between the corresponding spectral sequences.
\end{proof}

Since $\Phi(\x_{\xi|_N}) = \x_{\xi}$, Lemma~\ref{lem:morphism} implies that if
$\x_{\xi|_N}$ vanishes in $E^k(\mathcal{P}_{N},\arca_{N})$, then it also vanishes in $E^k(\mathcal{P},\arca)$.
Hence, by the definition of the spectral order,
\[
\AT(\mathcal{P}_{N},\underline{\mathbf{a}}_{N})\ge \AT(\mathcal{P},\underline{\mathbf{a}})\ge \AT(M,\xi).
\]
Taking the minimum of over all possible choices of $(\mathcal{P}_{N},\underline{\mathbf{a}}_{N})$,
we get that
\[
\AT(N,\xi|_{N})\ge \AT(M,\xi),
\]
as required. This concludes the proof of Theorem~\ref{thm:ineq}.
\end{proof}

We are now in a position to strengthen Lemma~\ref{lem:ball}.

\begin{cor}
Let $(M, \xi)$ be a connected contact $3$-manifold with (possibly empty) convex boundary,
and suppose that $B \subset \Int(M)$ is a tight contact ball. If $M_0 = M \setminus \Int(B)$
and $\xi_0 = \xi|_{M_0}$, then
\[
\AT(M_0,\xi_0) = \AT(M,\xi).
\]
\end{cor}

\begin{proof}
  We have already shown the closed case in Lemma~\ref{lem:ball},
  so we can suppose that $\partial M \neq \emptyset$. Let $M' \subset \Int(M)$
  be a codimension zero submanifold of $M$ with convex boundary,
  such that $(M',\xi')$ is contactomorphic to $(M,\xi)$, where $\xi' = \xi|_{M'}$. 
  Since $M$ is connected, we can assume that $B \subset \Int(M) \setminus M'$. 
  If we apply Theorem~\ref{thm:ineq} to the sequence $M' \subset M_0 \subset M$,
  we obtain that 
  \[
  \AT(M',\xi') \ge \AT(M_0,\xi_0) \ge \AT(M,\xi).
  \]
  As $(M,\xi)$ and $(M',\xi')$ are contactomorphic, $\AT(M,\xi) = \AT(M',\xi')$,
  and the result follows. 
\end{proof}

\section{Calculation of upper bounds on some spectral orders} \label{sec:gt}

Let $(M,\xi)$ be a contact $3$-manifold with convex boundary. Suppose that $(M,\xi)$
is overtwisted. Then, by definition, it contains an embedded overtwisted
disk $\Delta$. This has a standard neighborhood; i.e., there exists
a neighborhood $U\supset\Delta$ such that $(U,\xi|_{U})$ is contactomorphic
to a neighborhood of the disk $\Delta_{std}=\{z=0,\rho\le\pi\}$ inside
the standard overtwisted contact structure on $\mathbb{R}^{3}$, which is
defined as follows~\cite{key-8}:
\[
\xi_{OT}=\ker(\cos\rho \, dz+\rho\sin\rho\, d\phi).
\]
 Inside $U$, we can perturb $\Delta$ to a convex surface $D$. Take
a neighborhood $V=D\times [-1,1]$ such that $\xi|_{\text{Int}(M)}$ is $\mathbb{R}$-invariant.
After rounding its edges, we obtain an open subset $V_{0}\simeq D^3$
such that the dividing set $\Gamma_{V_{0}}$ on $\partial V_{0}$
is given by three disjoint curves. Honda, Kazez, and Mati\'c~\cite[Example~1]{key-1} gave a partial open
book decomposition of $N=\overline{V}_0$, and the corresponding Heegaard diagram is shown in Figure~\ref{fig:ot}.
\begin{figure}
\includegraphics{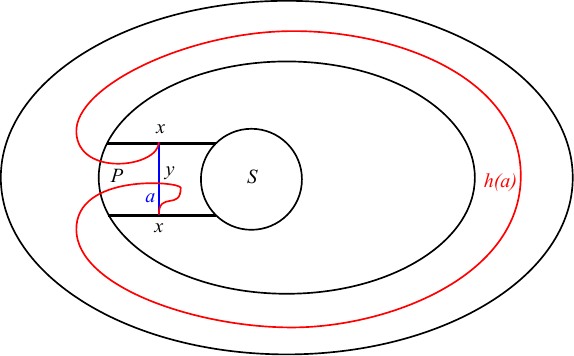}
\caption{A sutured Heegaard diagram arising from a partial open book decomposition of a neighborhood of an overtwisted disk.
We obtain the Heegaard surface by identifying the two bold horizontal arcs.}
\label{fig:ot}
\end{figure}
This diagram can be used to show that $\AT(M,\xi)=0$,
which was proven by Kutluhan et al.~\cite{key-10} in the closed case using the fact that an overtwisted
contact structure admits an open book whose monodromy is not right-veering.

\begin{rem}
It is convenient and customary to present the sutured diagram $(\S,\boldb,\bolda)$ 
arising from a partial open book decomposition $(S,P,h)$ and arcs basis $\arca$ 
on the surface $-S \times \{0\} \subset \S$.
Instead of gluing in $P \times \{0\}$, for each $a \in \arca$, we identify the opposite edges
of $N(a) \cap \partial S$ for a regular neighborhood $N(a)$ of $a$ in $S$. 
This is possible since $P = N(\arca)$.
\end{rem}

\begin{prop} \label{prop:ot}
If~$N$ is the standard neighborhood of an overtiwsted disk in the contact manifold $(M,\xi)$ as above,
then
\[
\AT(N,\xi|_{N})=0.
\]
\end{prop}

\begin{proof}
Honda, Kazez, and Mati\'c~\cite[Example~1]{key-1} computed that $c(N,\xi|_{N}) = 0$;
we extend their proof. Consider the partial open book decomposition
of $(N,\xi)$ shown in Figure~\ref{fig:ot}.
The contact element $\EH(N,\xi|_{N})$ is represented by the point~$x$,
which is zero in homology because $\partial y = x$. The only
J-holomophic curve from $y$ to $x$ is the bigon, which
satisfies $J_{+}=0$. Hence $\AT(N,\xi|_{N})\le0$.
\end{proof}

\begin{thm} \label{thm:OT}
If the contact manifold $(M,\xi)$ with convex boundary is overtwisted, then $\AT(M,\xi)=0$.
\end{thm}

\begin{proof}
We have $\AT(M,\xi)\le \AT(N,\xi|_{N})=0$ by Theorem~\ref{thm:ineq} and Proposition~\ref{prop:ot}.
\end{proof}

We now consider the case when $(M,\xi)$ has Giroux $2\pi$-torsion.
Recall that a contact manifold $(M,\xi)$ has $2\pi$-torsion if it
admits an embedding
\[
(M_{2\pi},\eta_{2\pi})=(T^{2}\times[0,1],\ker(\cos(2\pi t)\,dx-\sin(2\pi t)\,dy))\hookrightarrow(M,\xi).
\]
The boundary of $(M_{2\pi},\eta_{2\pi})$ is not convex. However, as in \cite[Lemma~5]{key-5},
if it embeds in $(M,\xi)$, then there exist small $\epsilon_0$, $\epsilon_1 >0$
such that the slightly extended domain
\[
(M',\eta')= \left(T^{2}\times[-\epsilon_0,1+\epsilon_1],\ker(\cos(2\pi t)\,dx-\sin(2\pi t)\,dy) \right)
\]
also embeds inside $(M,\xi)$ such that $T^2 \times \{-\epsilon_0\}$ and $T^2 \times \{\epsilon_1\}$
are pre-Lagrangian tori with integer slopes $s_0$ and $s_1$ that form a basis of $H_1(T^2)$.
By the work of Ghiggini~\cite{key-3}, we can perturb $\partial M'$
to get a new contact submanifold $\widetilde{M}$
such that $\partial\widetilde{M}$ is convex, and the slopes of the
dividing sets are $s_0$ and $s_1$. After a change of coordinates in $\widetilde{M}$,
we can assume these slopes are $0$ and $\infty$.

The contact manifold $\widetilde{M}$ is non-minimally-twisting and consists of five basic slices,
which means that we can construct a partial open book decomposition
of it by attaching four bypasses to a partial open book diagram of
a basic slice, which can be found in Examples~4, 5, and~6 of~\cite{key-1}.
The diagram we get is shown in Figure~\ref{fig:gt}.
\begin{figure}
\resizebox{.7\textwidth}{!}{%
\includegraphics{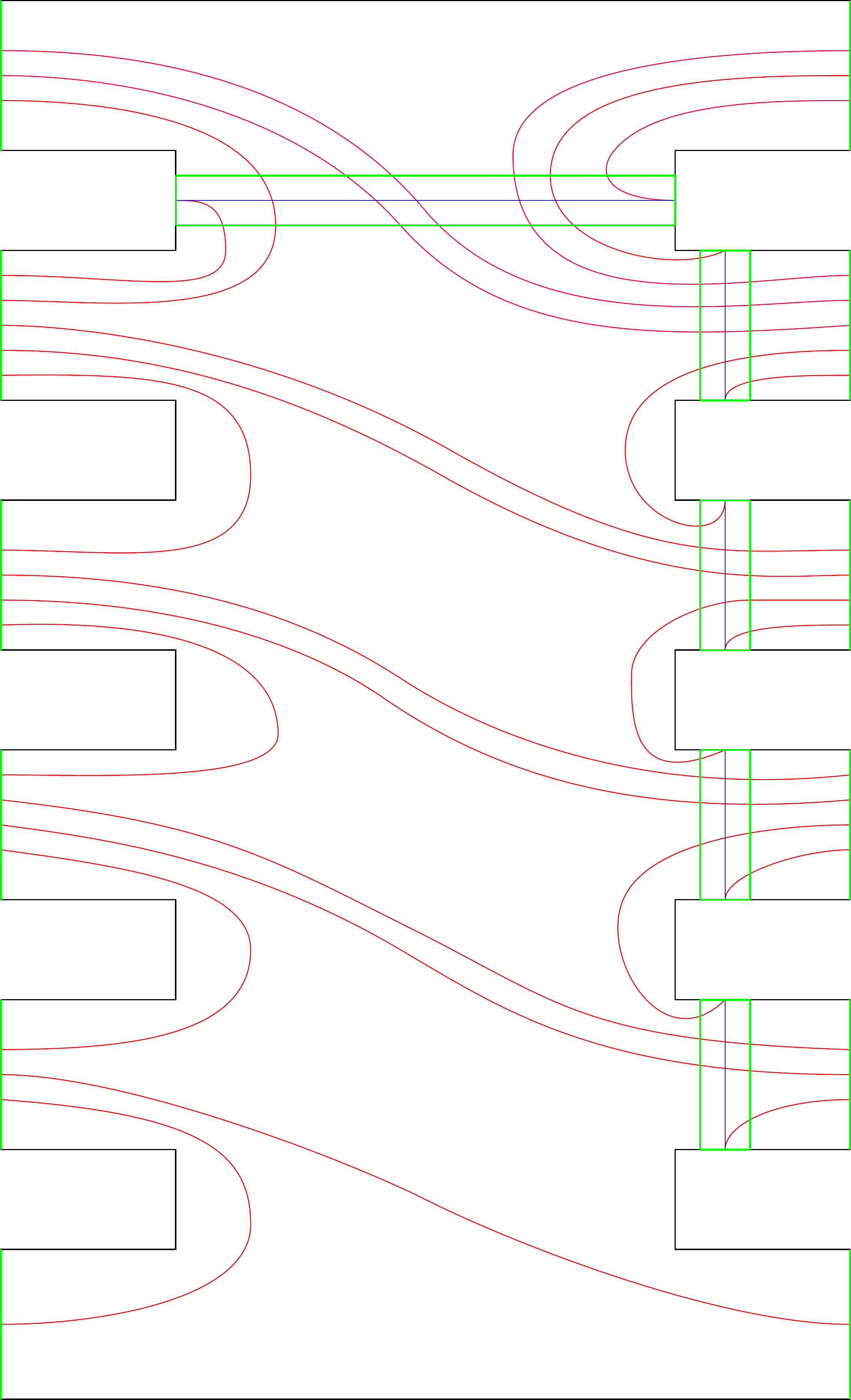}%
}
\caption{A sutured diagram arising from a partial open book decomposition of a neighborhood of a Giroux torsion domain.
The opposite green arcs in the boundary are identified.}
\label{fig:gt}
\end{figure}

\begin{figure}
\resizebox{.7\textwidth}{!}{%
\includegraphics{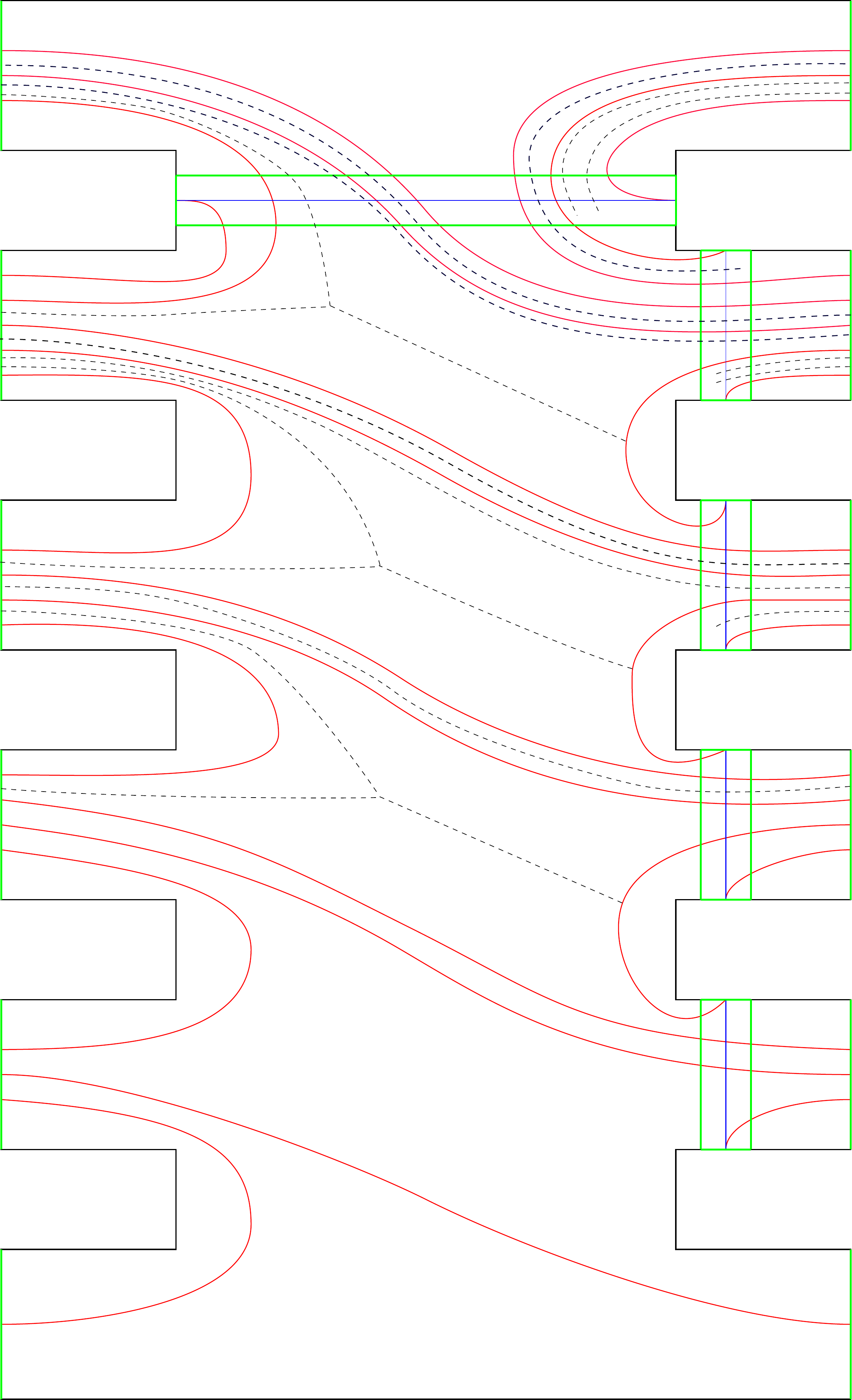}%
}
\caption{We apply the Sarkar-Wang algorithm by isotoping the red curves along the dashed arcs.}
\label{fig:dashed}
\end{figure}

Applying the Sarkar-Wang algorithm~\cite{key-12} to this diagram along the dotted arcs in Figure~\ref{fig:dashed},
we obtain the one in Figure~\ref{fig:SW}. It is easy to check that every
region that does not intersect the boundary is either a bigon or a quadrilateral.
In Figure~\ref{fig:SW}, the $\b$-curves are shown in red and the $\a$-curves in blue,
and the opposite green arcs in the boundary of the surface are identified.
The intersection points between $\bolda$ and $\boldb$ are labeled $x_1, \dots, x_{18}$
from right-to-left along the horizontal blue arc, and along the four vertical blue arcs
they are labeled from top-to-bottom $y_1, \dots, y_{15}$, $z_1, \dots, z_{10}$,
$w_1, \dots, w_6$, and $v_1, \dots, v_3$, respectively.
\begin{figure}
\resizebox{.7\textwidth}{!}{%
\includegraphics{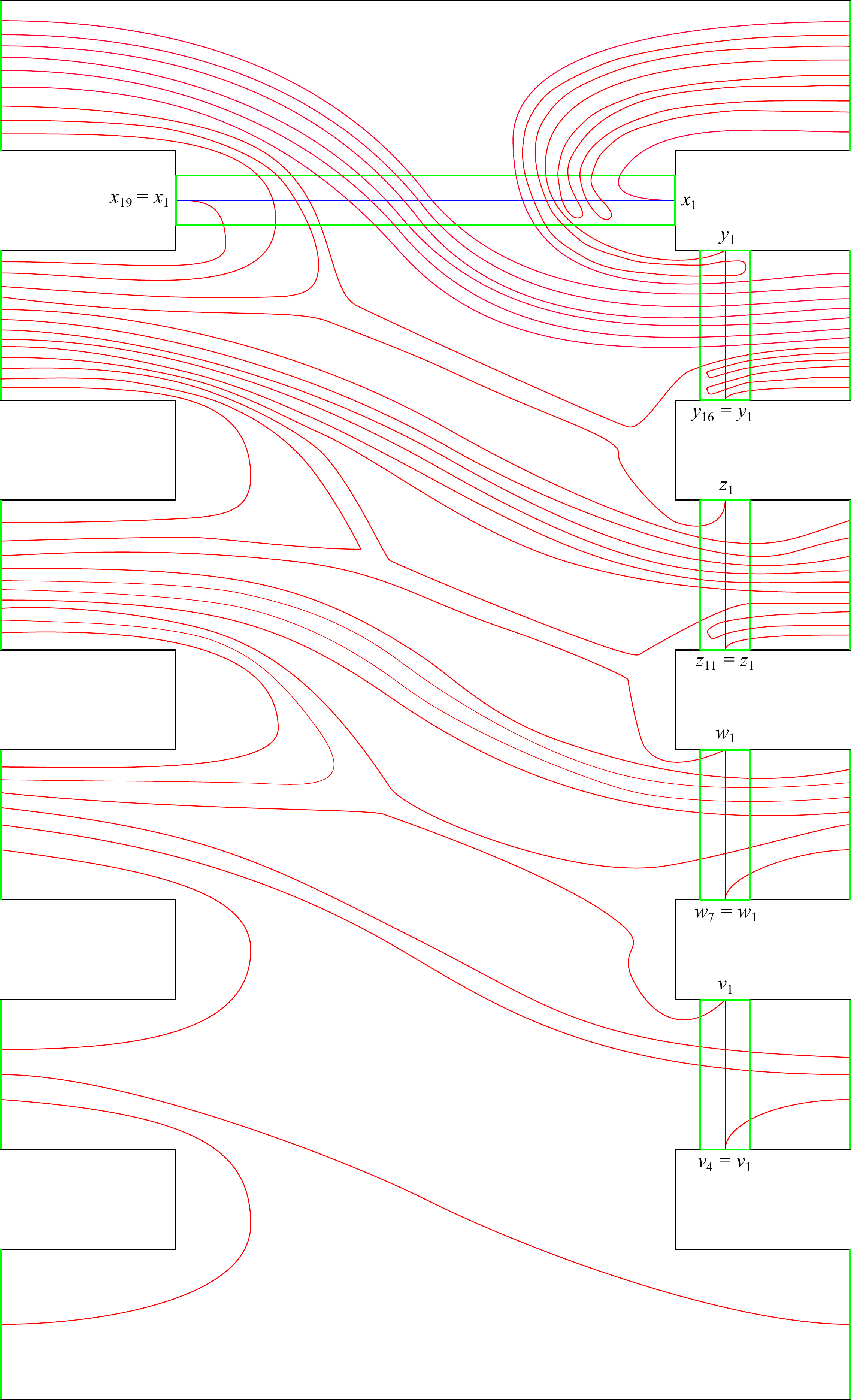}%
}
\caption{The diagram after applying the Sarkar-Wang algorithm.}
\label{fig:SW}
\end{figure}

\begin{figure}
\resizebox{.7\textwidth}{!}{%
\includegraphics{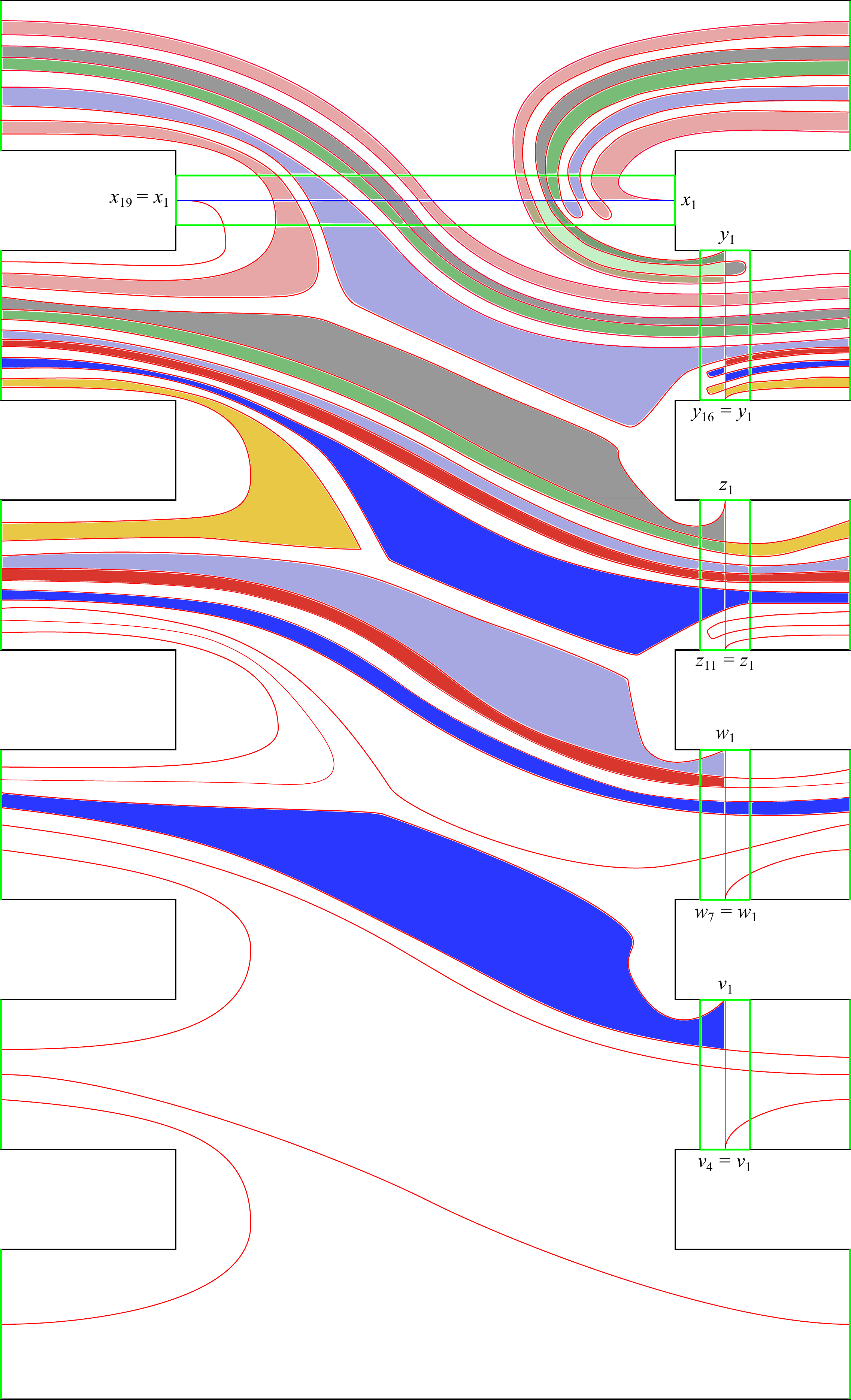}%
}
\caption{The quadrilaterals and bigons relevant to the computation are shaded.}
\label{fig:color}
\end{figure}

The contact element $\EH\left(\xi|_{\widetilde{M}}\right)$ is represented by the
unordered tuple $(x_{1},y_{1},z_{1},w_{1}, v_1)$. We now directly prove
that the contact invariant of $\widetilde{M}$ is zero and calculate
its spectral order with respect to the given diagram, thus giving
an upper bound on $\AT(\widetilde{M})$.

If $Q$ is a quadrilateral component of $\S \setminus (\bolda \cup \boldb)$ disjoint from $\partial \S$
with corners $c_1$, $c_2$, $c_3$, $c_4 \in \bolda \cap \boldb$,
then we say that $c_{1}$, $c_{3}$ are its from-corners and $c_{2}$, $c_{4}$
are its to-corners if
\[
\partial(\partial Q \cap \bolda) = c_1 + c_3 - c_2 - c_4.
\]
For any generator $(c_1,c_3,\dots) \in \T_\a \cap \T_\b$,
the coefficient of $(c_{2},c_{4},\dots)$ in the boundary $\partial(c_{1},c_{3},\dots)$
is the number of such quadrilaterals.

Since the only quadrilateral whose to-corners
are in $\{x_{1},y_{1},z_{1},w_{1},v_1\}$ is $y_{1}y_{2}z_{1}z_{2}$,
we get that
\[
\partial(x_{1},y_{2},z_{2},w_{1},v_1)=(x_{1},y_{1},z_{1},w_{1},v_1)+(x_{1},y_{3},z_{2},w_{1},v_1),
\]
where the last term comes from the bigon $y_{2}y_{3}$. This quadrilateral and bigon are  shaded
grey in Figure~\ref{fig:SW}.

The only quadrilateral whose to-corners are in $\{x_{1},y_{2},z_{2},w_{1},v_1\}$
is $x_{1}x_{2}y_{3}y_{4}$, and we have that
\[
\partial(x_{2},y_{4},z_{2},w_{1},v_1)=(x_{1},y_{3},z_{2},w_{1},v_1)+(x_{3},y_{4},z_{2},w_{1},v_1),
\]
where the last term comes from the bigon $x_{2}x_{3}$.
This quadrilateral and bigon are  shaded pink in Figure~\ref{fig:SW}.

The only quadrilateral whose to-corners are in $\{x_{3},y_{4},z_{2},w_{1},v_1\}$
is $x_{3}x_{4}w_{1}w_{2}$, and we have that
\[
\partial(x_{4},y_{4},z_{2},w_{2},v_1)=(x_{3},y_{4},z_{2},w_{1},v_1)+(x_{5},y_{4},z_{2},w_{2},v_1),
\]
where the last term comes from the bigon $x_{4}x_{5}$.
This quadrilateral and bigon are shaded light blue in Figure~\ref{fig:SW}.

The only quadrilateral whose to-corners are in $\{x_{5},y_{4},z_{2},w_{2},v_1\}$
is $x_{5}x_{6}z_{2}z_{3}$, shaded green in Figure~\ref{fig:SW}, and we have that
\[
\partial(x_{6},y_{4},z_{3},w_{2},v_1)=(x_{5},y_{4},z_{2},w_{2},v_1)+(x_9,y_1,z_3,w_2,v_1),
\]
where the last term comes from the quadrilateral $x_6x_9y_4y_1$.

The only quadrilateral whose to-corners are in $\{x_9,y_1,z_3,w_2,v_1\}$
is $y_1y_{15}z_{3}z_{2}$, and we have that
\[
\partial(x_{9},y_{15},z_{2},w_{2},v_1)=(x_9,y_1,z_3,w_2,v_1)+(x_9,y_{14},z_2,w_2,v_1),
\]
where the last term comes from the bigon $y_{15}y_{14}$.
This quadrilateral and bigon are shaded yellow in Figure~\ref{fig:SW}.

The only quadrilateral whose to-corners are in $\{x_9,y_{14},z_2,w_2,v_1\}$
is $y_{14}y_{13}v_1v_2$, and we have that
\[
\partial(x_{9},y_{13},z_{2},w_{2},v_2)=(x_9,y_{14},z_2,w_2,v_1)+(x_9,y_{12},z_2,w_2,v_2),
\]
where the last term comes from the bigon $y_{13}y_{12}$.
This quadrilateral and bigon are shaded blue in Figure~\ref{fig:SW}.

Finally, the only quadrilateral whose to-corners are in $\{x_9,y_{12},z_2,w_2,v_2\}$
is $y_{12}y_{11}w_2w_3$, shown in red, and we have that
\[
\partial(x_{9},y_{11},z_{2},w_{3},v_2)=(x_9,y_{12},z_2,w_2,v_2).
\]

Hence, over $\mathbb{F}_2$, 
\[
\begin{split}
\partial((x_{1},y_{2},z_{2},w_{1},v_1)+(x_{2},y_{4},z_{2},w_{1},v_1)+(x_{4},y_{4},z_{2},w_{2},v_1)+(x_{6},y_{4},z_{3},w_{2},v_1)\\
+ (x_{9},y_{15},z_{2},w_{2},v_1) + (x_{9},y_{13},z_{2},w_{2},v_2) + (x_{9},y_{11},z_{2},w_{3},v_2))=
(x_{1},y_{1},z_{1},w_{1}),
\end{split}
\]
which is exactly $\x_{\xi|_{\widetilde{M}}}$. Thus $\EH\left(\xi|_{\widetilde{M}}\right)=0$,
so the spectral order of $\widetilde{M}$ is finite.

\begin{rem*}
This result, together with the gluing map of \cite{key-2}, gives
an explicit computational proof of the fact that the contact invariant
of any contact $3$-manifold with Giroux $2\pi$-torsion vanishes,
which was proven in the closed case by Ghiggini, Honda, and Van Horn-Morris~\cite{key-5}.
\end{rem*}

\begin{rem*}
In \cite[Example 6-(c)]{key-1}, Honda, Kazez, and Mati\'c showed that if we only attach four bypasses to a basic slice;
i.e., if the contact structure is minimally twisting,
then the contact invariant is non-zero because it embeds in the unique Stein
fillable contact structure on~$T^{3}$, which already has non-zero
contact invariant. This can also be shown explicitly using a computation analogous to, but simpler
than the one above. Hence, it is necessary to enlarge the Giroux $2\pi$-torsion domain a bit
to obtain vanishing of the contact element.
\end{rem*}

\begin{prop} \label{prop:gt}
For the perturbed Giroux $2\pi$-torsion domain $\widetilde{M}$, we have
\[
\AT\left(\widetilde{M}, \xi|_{\widetilde{M}} \right)\le 2.
\]
\end{prop}

\begin{proof}
The complete list of the $J$-holomorphic disks used in the calculations above and
the values used to compute their $J_{+}$ are given in the table below.
If we label the $\a$- and $\b$-curves such that $x_1 \in \a_1 \cap \b_1$,
$y_1 \in \a_2 \cap \b_2$, $z_1 \in \a_3 \cap \b_3$, $w_1 \in \a_4 \cap \b_4$,
and $v_1 \in \a_5 \cap \b_5$, then
\[
\begin{split}
x_1, x_9, y_4, y_{16} \in \b_1, \\
x_6, y_1, z_2 \in \b_2, \\
z_1, x_2, x_3, y_2, y_3, y_{11}, w_2 \in \b_3, \\
x_4, x_5, y_{14}, y_{15}, z_3, w_1, v_2 \in \b_4, \\
y_{12}, y_{13}, w_3, v_1 \in \b_5.
\end{split}
\]
Furthermore,
$x_i \in \a_1$, $y_i \in \a_2$, $z_i \in \a_3$, $w_i \in \a_4$, and $v_i \in \a_5$
for any~$i$. Note that if there is a bigon connecting $\x$, $\y \in \T_\a \cap \T_\b$,
then $|\x| = |\y|$.
Using this,
\[
\begin{split}
|(x_1,y_1,z_1,w_1,v_1)| = |(1)(2)(3)(4)(5)| = 5, \\
|(x_1,y_2,z_2,w_1,v_1)| = |(1)(23)(4)(5)| = 4 = |(x_1,y_3,z_2,w_1,v_1)|, \\
|(x_2,y_4,z_2,w_1,v_1)| = |(132)(4)(5)| = 3 = |(x_3,y_4,z_2,w_1,v_1)|, \\
|(x_4,y_4,z_2,w_2,v_1)| = |(1432)(5)| = 2 = |(x_5,y_4,z_2,w_2,v_1)|, \\
|(x_6,y_4,z_3,w_2,v_1)| = |(12)(34)(5)| =  3, \\
|(x_9,y_1,z_3,w_2,v_1)| = |(1)(2)(34)(5)| = 4, \\
|(x_9,y_{15},z_2,w_2,v_1)| = |(1)(243)(5)| = 3 = |(x_9,y_{14},z_2,w_2,v_1)| \\
|(x_9,y_{14},z_2,w_2,v_2)| = |(1)(2543)| = 2 = |(x_9,y_{12},z_2,w_2,v_2)|, \\
|(x_9,y_{11},z_2,w_3,v_2)| = |(1)(23)(45)| = 3.
\end{split}
\]

From the table below, we see that every $J$-holomorphic disk $u$ used
to compute the differential satisfies $J_{+}(u) \le 2$; cf.~Equation~\ref{eqn:J}.

\begin{center}
\begin{tabular}{|c|c|c|c|c|}
\hline
Type & Name & $2(n_{\mathbf{x}}+n_{\mathbf{y}})$ & $|\mathbf{x}|-|\mathbf{y}|$ & $J_{+}$\tabularnewline
\hline
\hline
quadrilateral & $y_{1}y_{2}z_{1}z_{2}$ & 2 & -1 & 0\tabularnewline
\hline
quadrilateral & $x_{1}x_{2}y_{3}y_{4}$ & 2 & -1 & 0\tabularnewline
\hline
quadrilateral & $x_{3}x_{4}w_{1}w_{2}$ & 2 & -1 & 0\tabularnewline
\hline
quadrilateral & $x_{5}x_{6}z_{2}z_{3}$ & 2 & 1 & 2\tabularnewline
\hline
quadrilateral & $x_6 x_9 y_4 y_1$ & 2 & -1 & 0 \tabularnewline
\hline
quadrilateral & $y_1 y_{15} z_3 z_2$ & 2 & -1 & 0 \tabularnewline
\hline
quadrilateral & $y_{14} y_{13} v_1 v_2$ & 2 & -1 & 0 \tabularnewline
\hline
quadrilateral & $y_{12} y_{11} w_2 w_3$ & 2 & 1 & 2 \tabularnewline
\hline
bigon & $y_{2}y_{3}$ & 1 & 0 & 0\tabularnewline
\hline
bigon & $x_{2}x_{3}$ & 1 & 0 & 0\tabularnewline
\hline
bigon & $x_{4}x_{5}$ & 1 & 0 & 0\tabularnewline
\hline
bigon & $y_{15}y_{14}$ & 1 & 0 & 0\tabularnewline
\hline
bigon & $y_{13}y_{12}$ & 1 & 0 & 0\tabularnewline
\hline
\end{tabular}
\par\end{center}

For simplicity, we will write $(i,j,k,l,m)$ for the generator $(x_i,y_j,z_k,w_l,v_m)$.
Then let
\[
\begin{split}
b_0 &= (1,2,2,1,1) + (2,4,2,1,1) + (4,4,2,2,1) + (6,4,3,2,1) + (9,15,2,2,1) + (9,13,2,2,2), \\
b_1 &= (6,4,3,2,1) + (9,15,2,2,1) + (9,13,2,2,2) + (9,11,2,3,2), \text{ and} \\
b_2 &= (9,11,2,3,2),
\end{split}
\]
considered as chains with $\Z_2$ coefficients. Using the table above,
\[
\begin{split}
\partial_0 b_0 &= (1,1,1,1,1) + (5,4,2,2,1) + (9,12,2,2,2), \\
\partial_0 b_1 &= (9,12,2,2,2) \text{ and } \partial_1 b_1 = (5,4,2,2,1) + (9,12,2,2,2), \\
\partial_0 b_2 &= 0 \text{, } \partial_1 b_2 = (9,12,2,2,2), \text { and } \partial_2 b_2 = 0.
\end{split}
\]
Hence $\partial_0 b_0 + \partial_1 b_1 = (1,1,1,1,1)$, $\partial_0 b_1 + \partial_1 b_2 = 0$,
and $\partial_0 b_2 = 0$. So, if we set $b_i = 0$ for every $i > 2$,
then $\widehat{\partial}(b_i)_{i \in \N} = (c_i)_{i \in \N}$, where $c_0 = (1,1,1,1,1)$
represents the $\EH$ class, and $c_i = 0$ for every $i > 0$. By equation~\eqref{eqn:main},
the element $\x_{\xi|_{\widetilde{M}}}$ lies in $B^2_0$,
and hence vanishes in $E^3$; i.e., $\AT\left(\widetilde{M}, \xi|_{\widetilde{M}}\right)\le 2$,
as claimed.
\end{proof}

As an immediate corollary, we obtain Theorem~\ref{thm:gt} from the introduction.

\begin{cor}
If a contact $3$-manifold $(M,\xi)$ with convex boundary has Giroux $2\pi$-torsion,
then
\[
\AT(M,\xi) \le 2.
\]
\end{cor}

\begin{proof}
If the Giroux $2\pi$-torsion domain $M_{2\pi}$ embeds in $M$, then
the perturbed domain $\widetilde{M}$ also embeds in $M$, by the
argument outlined at the beginning of this section. Then Theorem~\ref{thm:ineq}
and Proposition~\ref{prop:gt} imply that
\[
\AT(M,\xi)\le \AT\left(\widetilde{M}, \xi|_{\widetilde{M}}\right) \le 2.
\]
\end{proof}

\section{Open questions}

We raise some questions that naturally arise from the discussions above.
First, as in the case of closed contact $3$-manifolds, we would like to
know how the spectral order $\AT(\mathcal{P},\underline{\mathbf{a}})$
depends on the choice of partial open book decomposition~$\mathcal{P}$ and
arc system~$\underline{\mathbf{a}}$.

\begin{rem*}
Given two possible choices of partial open book decompositions $(\mathcal{P},\arca)$
and $(\mathcal{P}',\arca')$ for a given
contact $3$-manifold $(M,\xi)$ with convex boundary, it is natural to ask whether
$\AT(\mathcal{P},\arca)=\AT(\mathcal{P}',\arca')$.
In the closed case, according to Kutluhan et al.~\cite{key-14},
the number $\AT(S,\phi,\underline{\mathbf{a}})$
does not depend on the isotopy class of the arc basis $\underline{\mathbf{a}}$,
but if two arc bases differ by an arc-slide, the corresponding values
of $\AT$ might not be the same. Since our definition of $\AT$ is a direct
generalization of the original one, the same holds in our case.
\end{rem*}

Now, given the inequality $\AT(N,\xi|_{N}) \ge \AT(M,\xi)$,
whenever $(N,\xi|_{N})$ is a compact codimension zero submanifold
of $(M,\xi)$ with convex boundary, we are led to the following question.

\begin{qn} \label{qn:1}
If a contact $3$-manifold $(N,\xi)$ with convex boundary satisfies
$\AT(M,\xi_{M}) \le k$ for every closed contact $3$-manifold
$(M,\xi_{M})$ in which $(N,\xi)$ embeds, do we have $\AT(N,\xi)\le k$?
\end{qn}

An affirmative answer to Question~\ref{qn:1} would imply that the inequality
$\AT(N,\xi|_{N}) \le \AT(M,\xi)$ is sharp
and cannot be improved without giving extra conditions even when~$M$
is assumed to be closed. We can ask the following question
regarding the spectral order of planar torsion domains.

\begin{qn} \label{qn:2}
Is there a way to prove that the order of a Giroux torsion domain is at most~$1$, instead of~$2$?
\end{qn}

The upper bound to the spectral order of a Giroux torsion domain is predicted to be $1$ by \cite[Question~6.3]{key-14}, since a Giroux torsion domain is a planar torsion domain of order $1$. However, our computation only allows us to prove that it is at most $2$. If the above question has an affirmative answer, then we must be able to prove it via explicit computation by starting from a complete system of arcs, and then duplicating the arcs, one by one. The problem is that the resulting diagram is too large for practical computation by hand.

Finally, probably the most interesting question in this area is whether
the converse of Theorem~\ref{thm:OT} holds, analogously to \cite[Question~6.1]{key-14}.

\begin{qn}
If $\AT(M,\xi) = 0$, then does this imply that $\xi$ is overtwisted?
\end{qn}

\end{document}